\documentclass[11pt]{article}
\usepackage{amsmath,amsfonts,amssymb,latexsym,amsthm,dsfont}
\usepackage[a4paper,margin=2.5cm]{geometry}
\usepackage{graphicx}
\usepackage[backref]{hyperref}
\usepackage[utf8]{inputenc}
\usepackage[T1]{fontenc}
\usepackage{lmodern}
\usepackage{enumitem}
\usepackage{tikz}
\usetikzlibrary{arrows,calc,decorations.text}
\usetikzlibrary{external}

\usepackage[english]{babel}

\newtheorem{thm}{Theorem}[section]%
\newtheorem{cor}[thm]{Corollary}%
\newtheorem{prop}[thm]{Proposition}%
\newtheorem{lem}[thm]{Lemma}%

\newtheorem{xpl}[thm]{Example}%
\newtheorem{ass}[thm]{Assumption}%
\newtheorem{rem}[thm]{Remark}%
\newcommand{\dP}{\mathbb{P}}
\newcommand{\dR}{\mathbb{R}}

\newcommand{\cC}{\mathcal{C}}

\newcommand{\ABS}[1]{{{\left| #1 \right|}}} 
\let\abs\ABS
\newcommand{\BRA}[1]{{{\left\{#1\right\}}}} 
\newcommand{\DP}[1]{{{\left<#1\right>}}} 
\newcommand{\NRM}[1]{{{\left\| #1\right\|}}} 
\newcommand{\PAR}[1]{{{\left(#1\right)}}} 
\newcommand{\SBRA}[1]{{{\left[#1\right]}}} 
\renewcommand{\leq}{\leqslant}
\renewcommand{\geq}{\geqslant}

\newcommand{\ind}{\mathds{1}}

\newcommand{\torus}{\mathbb{T}}

\title{On the stability of planar randomly switched systems}

\author{
Michel~\textsc{Benaïm}, 
Stéphane~\textsc{Le Borgne}, 
Florent~\textsc{Malrieu}
and Pierre-Andr\'e~\textsc{Zitt}} %

\date{\today}

\begin{document}

\maketitle

\begin{abstract}
Consider the random process ${(X_t)}_{t\geq 0}$ solution of 
$\dot{X}_t=A_{I_t} X_t$ where ${(I_t)}_{t\geq 0}$ is a Markov process 
on $\{0,1\}$ and $A_0$ and $A_1$ are real Hurwitz matrices on $\dR^2$.  
Assuming that there exists $\lambda\in (0,1)$ such that 
$(1-\lambda)A_0+\lambda A_1$ has a positive eigenvalue, we establish 
that $\|X_t\|$ may converge to 0 or $+\infty$ depending on the the jump rate of 
the process $I$. An application to product of random matrices is studied. 
This paper can be viewed as a probabilistic counterpart of the paper \cite{1} 
by Balde, Boscain and Mason.  
\end{abstract}

{\footnotesize %

\noindent\textbf{AMS-MSC.}  60J75; 60J57; 93E15; 34D23 %
} 

\section{Introduction}

The motivation of the present paper is twofold. Firstly, this work
answers a question by G.~Charlot about the stochastic counterpart 
of the work~\cite{1}. Secondly,  the Piecewise Deterministic Markov 
processes (PDMP) under study may present a surprising blow-up 
when time goes to infinity. 

Let $A_0,A_1\in \dR^{2\times 2}$ be two real matrices which admit
two eigenvalues with negative real parts: $A_0$ and $A_1$ 
are said to be Hurwitz matrices. In \cite{1}, the authors deal with the stability 
problem for the planar linear switching system 
$\dot{x}_t=(1-u_t)A_0x_t+u_tA_1x_t$, where  
$u$: $[0,\infty)\to \BRA{0,1}$ is a measurable function. 
They provide necessary and sufficient conditions on $A_0$ and 
$A_1$ for the system to be asymptotically stable for arbitrary 
switching function~$u$. The main hypothesis that ensures 
the existence of a control $u$ such that the system is not 
asymptotically stable is the following.

\begin{ass}
  \label{ass:lambda}
  There exists~$\lambda\in(0,1)$ such that the matrix 
  $A_\lambda=(1-\lambda) A_0+\lambda A_1$ has two real eigenvalues 
  $ - \lambda_- < 0 < \lambda_+$ with opposite signs. Let us denote by 
  $u_-$, $u_+$ two associated (real, unit) eigenvectors.  
\end{ass}

\begin{rem}
 It is shown in \cite{1} that Assumption~\ref{ass:lambda} is equivalent 
 to the relation 
\begin{equation}\label{eq:cond-BBM}
\mathrm{Tr}(A_0)\mathrm{Tr}(A_1)-\mathrm{Tr}(A_0A_1)
<-2\sqrt{\det(A_0)\det(A_1)}.
\end{equation}
\end{rem}

  Assumption~\ref{ass:lambda} may hold in many different cases as it is 
  illustrated by the two following Examples~\ref{ex:jordan} 
  and~\ref{ex:rotations}. The complete description of the different 
  cases is postponed to Section~\ref{sec:angular}.

\begin{xpl}\label{ex:jordan}
Let us define $A_0$ and $A_1$ by 
\[
A_0=\PAR{
\begin{array}{cc}
-1 & 2b \\
0  & -1 
\end{array}
}
\quad\text{and}\quad
A_1=\PAR{
\begin{array}{cc}
-1 & 0 \\
2b  & -1 
\end{array}
}
\]
with $b>0$. Then $A_0$ and $A_1$ are two Jordan matrices and 
the eigenvalues of $A_{1/2}$ are given by $-1\pm b$.  
\end{xpl}

\begin{xpl}\label{ex:rotations}
Let us define $A_0$ and $A_1$ by 
  \[
A_0=\PAR{
\begin{array}{cc}
-1 & ab \\
-a/b  & -1 
\end{array}
}
\quad\text{and}\quad
A_1=\PAR{
\begin{array}{cc}
-1 & -a/b \\
ab  & -1 
\end{array}
}
\]
with $a,b>0$. Then $A_0$ and $A_1$ have conjugate complex eigenvalues and 
the eigenvalues of $A_{1/2}$ are $-1\pm a (b-1/b)/2$. 
\end{xpl}

In the sequel, we suppose that Assumption~\ref{ass:lambda} holds. Let us define 
$\lambda_0=\lambda$ and $\lambda_1=1-\lambda$. For any $\beta>0$, 
consider the Markov process $(X,I)$ on $\dR^2 \times \{0,1\}$ driven by the 
generator~$\mathcal{L}_\beta$:
\[
  \mathcal{L}_\beta f(x,i) = \mathcal{L}_C f (x,i) + \beta \mathcal{L}_J f(x,i) 
\]
where 
\[
  \mathcal{L}_C f(x,i)=A_i \nabla f(x,i) 
  \quad\text{and}\quad
  \mathcal{L}_J f(x,i)=\lambda_i(f(x,1-i) - f(x,i)).
\]
The operator $\mathcal{L}_C$ corresponds to the ``continuous'' part (the 
first component $x$ evolves along the flow of the vector field $x\mapsto A_ix$) 
and $\beta \mathcal{L}_J$ gives the jumps on the second component. 
If $\nu$ is a probability measure on $\dR^2 \times \{0,1\}$, we denote by
$\dP_\nu$ the law of the process $(X,I)$ when the law of $(X_0,I_0)$ is~$\nu$. 
\begin{rem}
One can easily construct the process $(X,I)$ as follows. The 
process~${(I_t)}_{t\geq 0}$ is the Markov process on $\BRA{0,1}$ with 
jump rates ${(\beta \lambda_i)}_{i\in\BRA{0,1}}$. Then, ${(X_t)}_{t\geq 0}$ is 
the solution of 
\[
X_t=X_0+\int_0^t \! A_{I_s}X_s\,ds, \quad (t\geq 0).
\]
Notice that ${(I_t)}_{t\geq 0}$ is a Markov process with invariant measure 
\[
\frac{\beta\lambda_1}{\beta\lambda_0+\beta\lambda_1}\delta_0
+\frac{\beta\lambda_1}{\beta\lambda_0+\beta\lambda_1}\delta_1
=(1-\lambda)\delta_0+\lambda\delta_1. 
\]
\end{rem}
Our main result ensures that under Assumption~\ref{ass:lambda} the norm of 
the continuous component $X$ goes to zero if the jumps are rare and to $+\infty$ 
if the jumps are sufficiently numerous (and $X_0\neq 0$). 
\begin{thm}
  \label{thm:mainResult}
  Under Assumption~\ref{ass:lambda}, there exists $\chi(\beta)\in\dR$ such that, 
  for any initial measure $\nu$ such that $\nu (\{0\}\times \BRA{0,1}) = 0$,
\begin{equation}
  \label{eq:generic-convergence}
\frac{1}{t}\log \NRM{X_t}\xrightarrow[t\to\infty]{\dP_\nu-a.s.}\chi(\beta).  
\end{equation}
Moreover, there exist two constants $0<\beta_1 \leq  \beta_2< \infty$ such that: 
  \begin{itemize}
    \item if $\beta < \beta_1$, then $\chi(\beta)$ is negative and 
    $\NRM{X_t} \xrightarrow[t\to \infty]{\dP_\nu-a.s.} 0$, 
    \item if $\beta > \beta_2$, then $\chi(\beta)$ is positive and 
    $\NRM{X_t} \xrightarrow[t\to \infty]{\dP_\nu-a.s.} \infty$.
  \end{itemize}
\end{thm}

 \begin{rem}
 The process ${((X_t,I_t))}_{t\geq 0}$ is what is called a Piecewise Deterministic 
 Markov Process on $\dR^2\times\BRA{0,1}$ (see \cite{davis,jacobsen} 
 for details) where the continuous part is driven by two vectors 
 fields that admit a unique stable point and are exponentially 
 stable. In \cite{BH} it is proved that if the process is recurrent  its invariant measure is often absolutely 
 continuous. The previous theorem shows that the recurrence may not be so easy to establish (it can depend on the jump rates).
\end{rem}

We prove Theorem~\ref{thm:mainResult} in Section \ref{sec:mainResult}. 
We do not know if $\beta_1=\beta_2$ under Assumption~\ref{ass:lambda}. 
Nevertheless, Section~\ref{sec:example} is dedicated to the study of  
Examples~\ref{ex:jordan} and~\ref{ex:rotations} where this "phase transition" 
can be established. The exponential rate 
of growth of the process is given by an expression analogous to Furstenberg 
formula (\cite{3}). Generally it is difficult to compute the element entering 
the Furstenberg formula (see examples in \cite{2}, \cite{4}). For the example of 
Section~\ref{sec:example} one obtains an explicit expression of the 
"Lyapunov" exponent of ${(X_t)}_{t\geq 0}$. Finally, in 
Section~\ref{sec:matrices}, we remark that our results can be interpreted 
in terms of products of random matrices. We obtain examples of products 
of random independent matrices, all of them contracting, with a positive 
Lyapunov exponent (we are not in the frame of unimodular matrices 
studied in \cite{2}, \cite{4}).

\section{The general case}
\label{sec:mainResult}
The proofs of the two parts of Theorem~\ref{thm:mainResult} use 
different techniques. The easy part, when $\beta$ is small,  follows 
from a martingale argument explained in Section \ref{sec:fewJumps}. 
To study the process for large~$\beta$, we use a polar decomposition, 
detailed in Section~\ref{sec:polar}. The angular process
is studied in Sections~\ref{sec:angular} and \ref{sec:ergodicAngular}. In Section~\ref{sec:blowup} we give the 
main line of the proof of Theorem~\ref{thm:mainResult}; the proof of a key 
lemma is postponed to Section~\ref{sec:concentration}. 

\subsection{Few jumps: convergence to zero}
\label{sec:fewJumps}

 In this subsection, we suppose that $\beta$ is small: the $i$ component 
rarely jumps. The two flows associated to $A_0$ and $A_1$ being linear 
and attractive, there exists $\rho>0$ and two norms $V_0$ and $V_1$, 
given by two positive symmetric matrices $M_0$ and $M_1$, such that, for 
$V_i(x) = \langle x, M_i x\rangle$, 
\[ 
\mathcal{L}_CV_i(x,i) \leq -\rho V_i(x). 
\]
Define, $V(x,i) = V_i(x)$. Since 
$\ABS{\mathcal{L}_J f(x,i)}\leq K(\ABS{f(x,0)} + \ABS{f(x,1)})$, we get  
\begin{align*}
\mathcal{L}_\beta V(x,i) 
&= \mathcal{L}_CV_i(x,i) + \beta \mathcal{L}_J V_i(x,i) \\
&\leq -\rho V_i(x) + \beta K (V_0(x) + V_1(x)) \\
&\leq -\rho V_i(x) + \beta K' V_i(x)
\end{align*}
by the equivalence of the norms. Therefore  there exist a $\rho'>0$ 
and a $\beta_1>0$ such that, for $\beta< \beta_1$, 
\[
\forall (x,i)\in\dR^2\times\BRA{0,1}, \quad 
\mathcal{L}_\beta V(x,i)  \leq - \rho'V(x,i). 
\]
Consequently the process ${(M_t)}_{t\geq 0}$ defined by 
$M_t = e^{\rho't} V(X_t,I_t)$ is a positive supermartingale. It converges 
almost surely to a random variable which is almost surely finite. 
Therefore $V(X_t,I_t)$ converges almost surely to zero, and $\NRM{X_t}$ 
itself converges to zero almost surely (exponentially fast).  

\subsection{A polar decomposition}
\label{sec:polar}

We begin by decomposing the deterministic dynamics.
Let $A$ be a matrix on $\dR^2$ and $x\in\dR^2\backslash \BRA{0}$.
Consider ${(x_t)}_{t\geq 0}$ the solution of 
\[
\begin{cases}
 \dot{x}_t=Ax_t,\\
 x_0=x.
\end{cases}
\]

First of all, since $x$ is not $0$, then, for any $t\geq 0$, $x_t$  is not equal to $0$. 
Therefore it is possible to define the polar coordinates $(r_t,\theta_t)$ of 
$x_t$. Call $e_\theta$ the unit vector $(\cos\theta, \sin\theta)$ 
and define $u_t = e_{\theta_t}$: $x_t$ may be written $r_tu_t$. Since 
$r_t^2 = \DP{x_t,x_t}$, we have: 
\begin{align*}
  r_t \dot{r}_t &= \DP{x_t,Ax_t} \\
  A(r_tu_t) &=  \dot{x}_t = \dot{r}_t u_t + r_t \dot{u}_t.
\end{align*}
Therefore:
\begin{align}
  \label{eq:evolR}
   \dot{r}_t &= r_t \DP{u_t,Au_t} \\
   \label{eq:evolU}
   \dot{u}_t &= Au_t  - \DP{u_t,Au_t} u_t. 
\end{align}
The evolution of $u_t$ on the circle is autonomous. The derivative $ \dot{u}_t$ vanishes when $Au_t=\DP{u_t,Au_t} u_t$ that is when $u_t$ is a eigenvector of $A$. As 
a consequence, the equation  (\ref{eq:evolU}) has
\begin{itemize}
\item four stationary points iff $A$ admits two different eigenvalues, 
\item two stationary points iff $A$ is a 
  Jordan matrix as in 
  Example~\ref{ex:jordan},
\item no stationary points iff the eigenvalues of $A$ are not real. 
\end{itemize}
If we write equation (\ref{eq:evolU}) in 
terms of the angles $\theta_t$. Since $\dot{u}_t = \dot{\theta_t} e_{\theta_t + \pi/2}$, 
the scalar product of \eqref{eq:evolU} with
$e_{\theta_t+ \pi/2}$ gives:
\begin{align}
  \dot{\theta}_t &= \DP{ Ae_{\theta_t}, e_{\theta_t + \pi/2}} \notag\\
  &= (A_{22} - A_{11} )\sin(\theta_t)\cos(\theta_t) 
  + A_{21}\cos^2(\theta_t) - A_{12}\sin^2(\theta_t).  
  \label{eq:evolTheta}
\end{align}
The critical points of this differential equation are related to the eigenvector of $A$
as it is pointed out in the following lemma. 
\begin{lem}
   \label{lem:classif}
For any matrix $A$, the function 
\[
d:\ \theta\mapsto d(\theta)= \DP{ Ae_{\theta}, e_{\theta + \pi/2}}
\]
given by \eqref{eq:evolTheta} is $\pi$-periodic and $d(\theta)=0$ iff 
$e_\theta$ is an eigenvector of $A$. Finally, the function $d$ is constant and equal to zero 
iff $A=\lambda\mathrm{I}_2$. 
\end{lem}
\begin{proof}
If $\theta$ is changed to $\theta+\pi$ then both $e_{\theta}$ and $e_{\theta+\pi/2}$ are changed to their opposite, so that  $\DP{ Ae_{\theta}, e_{\theta + \pi/2}}$ remains unchanged. We have already seen that $d(\theta)=0$ if and only if $e_\theta$ is an eignevector of $A$. 
\end{proof}

\subsection{The angular process}
\label{sec:angular}
Let us use the polar decomposition to study the 
process~${((X_t,I_t))}_{t\geq 0}$. Between jumps, 
the process follows the deterministic dynamics described above, 
with $A \in\BRA{A_0,A_1}$. Since the evolution of the angle $\theta$ 
is autonomous for each dynamics, the process $(\Theta,I)$ 
is a Markov process on $\dR\times \BRA{0,1}$. The evolution 
of ${(R_t)}_{t\geq 0}$ is determined by the one of the 
process~${((\Theta_t,I_t))}_{t\geq 0}$, by solving 
Equation~\eqref{eq:evolR} between the jumps. If we call 
$\mathcal{A}(\theta,i) = \DP{A_i e_\theta, e_\theta}$, then
\begin{equation}
  \label{eq:rEnFonctionDeTheta}
  R_t = R_0 \exp\left( \int_0^t \mathcal{A}(\Theta_s,I_s) ds \right). 
\end{equation}
and $R_t$ appears as a multiplicative functional of 
${((\Theta_s,I_s))}_{0\leq s\leq t}$. 

The proof of Theorem~\ref{thm:mainResult} relies on the study of the 
long time behavior of $(\Theta,I)$. We will see in the sequel that this process 
may be ergodic (\emph{i.e.} it may admits a unique invariant measure) or 
not. Let us define, for $i\in\{0,1\}$ and $\lambda\in(0,1)$, 
\begin{align*}
  d_i(\theta)       &= \DP{A_i e_\theta, e_{\theta + \pi/2}}, \\
  d_\lambda(\theta) &= (1- \lambda) d_0(\theta) + \lambda d_1(\theta). 
\end{align*}
The generator of the Markov process $(\Theta,I)$  is given by:
\[
L_\beta f(\theta,i)  = L_C f (\theta,i) + \beta L_Jf(\theta,i) 
\]   
where
\begin{equation}
 \label{eq:def-LC-LJ}
L_Cf(\theta,i) = d_i(\theta) \partial_\theta f(\theta,i)
\quad\text{and}\quad
L_Jf(\theta,i) = \lambda_i (f(\theta,1-i) - f(\theta,i)).
\end{equation}
Once again, $L_C$ is the continuous drift and $\beta L_J$ is the jump part. 
Let us also introduce the averaged (deterministic) dynamic:
\[
L_A f(\theta,i) = d_\lambda (\theta) \partial_\theta f(\theta,i).
\]
Under Assumption~\ref{ass:lambda}, Lemma~\ref{lem:classif} ensures 
that the vector field~$F^\lambda=d_\lambda \partial_\theta$ has exactly 
four critical points on $[0,2\pi)$. As $d_\lambda$ is $\pi$-periodic it suffices to describe it only on an interval of length $\pi$ separating two zeros of $d_\lambda$ corresponding to the negative eigenvalues of $A_\lambda$. Let $[\theta_-,\theta_-+\pi)$ this interval. The function $d_\lambda$ vanishes only once on $(\theta_-,\theta_-+\pi)$ at a point $\theta_+$ correponding to the positive eigenvalues of $A_\lambda$. We have 
\begin{equation}
 \label{eq:signe-dlambda}
d_\lambda(\theta)
\begin{cases}
 >0 &\text{if } \theta\in(\theta_-,\theta_+),\\
 <0 &\text{if } \theta\in(\theta_+,\theta_-+\pi).
\end{cases}
\end{equation}
Let us firstly notice that, under Assumption~\ref{ass:lambda}, the critical 
points $d_0$, $d_1$ and $d_\lambda$ are different.
\begin{lem}\label{lem:critic}
Under Assumption~\ref{ass:lambda} if $\theta$ is a critical point of $d_\lambda$
then $d_0(\theta)d_1(\theta)<0$. In particular, $\theta$ is not a critical point 
of $d_i$, $i\in\BRA{0,1}$.
\end{lem}
\begin{proof}
Assume that there exists $\theta$ such that 
$d_\lambda(\theta)=0=d_0(\lambda)$. Then $d_1(\theta)=0$. As a consequence, 
$u_\theta$ is an eigenvector for $A_0$, $A_1$ and $A_\lambda$ associated 
to the respective eigenvalues $\eta_0$, $\eta_1$ and $\eta_\lambda$. 
By definition, $\eta_\lambda=(1-\lambda)\eta_0+\lambda \eta_1$. This implies 
that the second eigenvalue of $A_\lambda$ is also a convex combination 
of two complex numbers with negative real part (consider the relation 
$\mathrm{Tr}(A_\lambda)
=(1-\lambda)\mathrm{Tr}(A_0)+\lambda \mathrm{Tr}(A_0)$). This cannot 
hold under Assumption~\ref{ass:lambda}. As a consequence, 
$d_0(\theta)d_1(\theta)\neq 0$. Since $d_\lambda(\theta)=0$, we get 
that $d_0(\theta)$ and $d_1(\theta)$ have opposite signs. 
\end{proof}

Without loss of generality we can assume that $d_0(\theta_+)<0$ and $d_1(\theta_+)>0$.
Because of the equality $d_\lambda=(1- \lambda) d_0(\theta) + \lambda d_1(\theta)$ we have constraints on the signs of the $d_i$.
Let us list all the possibilities:
\begin{enumerate}[label=(\alph*)]
  \item $d_1$ does not vanish and $d_0$ vanishes 0, 1 or 2 times on $(\theta_-,\theta_+)$,
  \item $d_0$ does not vanish and $d_1$ vanishes 0, 1 or 2 times on $(\theta_+,\theta_-+\pi)$,
  \item $d_1$ vanishes 2 times on  $(\theta_+,\theta_-+\pi)$ at points $\theta_{1m}<\theta_{1M}$ and $d_0$ vanishes 1 or 2 times on $(\theta_{1m},\theta_{1M})$,
  \item $d_0$ vanishes 2 times on  $(\theta_-,\theta_+)$ at points $\theta_{0m}<\theta_{0M}$ and $d_1$ vanishes 1 or 2 times on $(\theta_{0m},\theta_{0M})$,
  \item $d_1$ vanishes 1 or 2 times on $(\theta_+,\theta_-+\pi)$ at points $\theta_{1m}\leq \theta_{1M}$ and $d_0$ vanishes 1 or 2 times on  $(\theta_-,\theta_+)$ at points $\theta_{0m}\leq\theta_{0M}$,
  \item $d_0$  vanishes  2 times at points $\theta_{0m}<\theta_{0M}$ and $d_1$  vanishes  2 times at points $\theta_{1m}<\theta_{1M}$ such that $\theta_{1m}<\theta_{0m}<\theta_+<\theta_{1M}<\theta_{0M}$.
\end{enumerate}

In the last two cases we have a subinterval of $(\theta_-,\theta_-+\pi)$ that is invariant for both of the systems $\dot{\theta_t}=d_i(\theta_t)$ : $(\theta_{0M},\theta_{1m})$ in case (e), $(\theta_{0m},\theta_{1M})$ in case (f) (see Figure~\ref{fig=cases}).

\begin{figure}
{
\def\tzerom{30}%
\def\tzeroM{50}%
\def\tunm{120}%
\def\tunM{155}%
\def\tm{0}%
\def\tp{94}%
\def\tzerodemi{40}
\def\tundemi{137.5}

\noindent\begin{tikzpicture}[%
  scale=2.4,
  mean/.style={blue!50!red},
  zero/.style={blue!80!black},
  one/.style={red!80!black},
  angle/.style={fill=white,fill opacity = 0.7,text opacity = 1,inner sep=1pt,font={\small}},
  flow/.style={)->,very thick,shorten <= 1pt, shorten >= 1pt},
  alongflow/.style={flow,thin},
  ]
  \draw[very thin] (0,0) circle (1cm);
  \begin{scope}[zero]
    \draw (\tzerom + 180:1.3) -- (\tzerom:1.5);
    \draw (\tzeroM + 180:1.3) -- (\tzeroM:1.5);
    \draw[flow] (\tzerom:1.15)  arc[radius=1.15cm,start angle=\tzerom,end angle=\tzeroM];
    \draw[flow] (\tzerom:1.15)  arc[radius=1.15cm,start angle=\tzerom,end angle=\tzerom - 40];
    \draw[flow] (\tzerom + 180 :1.15)  arc[radius=1.15cm,start angle=\tzerom+180,end angle=\tzeroM ];
    \draw[flow,->]  (\tzerodemi:1.15) -- ++(\tzerodemi + 90:1pt);
    \draw[flow,->]  (\tundemi:1.15)   -- ++(\tundemi - 90:1pt); 
    \node[angle,right] at (\tzerom:1.55) {$\theta_{0m}$};
    \node[angle,right] at (\tzeroM:1.55) {$\theta_{0M}$};
  \end{scope}
  \begin{scope}[one]
    \draw (\tunm + 180:1.25) -- (\tunm:1.5);
    \draw (\tunM + 180:1.3) -- (\tunM:1.5);
    \draw[flow,->] (-10:1.3)  arc[radius=1.3cm,start angle=-10,end angle=\tunm];
    \draw[flow] (\tunM:1.3)  arc[radius=1.3cm,start angle=\tunM,end angle=\tunm];
    \draw[flow] (\tunM:1.3)  arc[radius=1.3cm,start angle=\tunM,end angle=\tunM + 40];
    \draw[flow,->]  (\tzerodemi:1.3) -- ++(\tzerodemi + 90:1pt);
    \draw[flow,->]  (\tundemi:1.3)   -- ++(\tundemi - 90:1pt); 
    \node[angle,left] at (\tunm:1.55) {$\theta_{1m}$};
    \node[angle] at (\tunM:1.55) {$\theta_{1M}$};
  \end{scope}
  \begin{scope}[mean]
    \draw (\tm + 180:1.5) -- (\tm:1.5);
    \draw (\tp + 180:1.2) -- (\tp:1.45);
    \draw[flow] (\tm:0.85)  arc[radius=0.85cm,start angle=\tm ,end angle=\tp];
    \draw[flow] (\tm+180:0.85)  arc[radius=0.85cm,start angle=\tm+180,end angle=\tp];
    \draw[flow,->]  (\tzerodemi:0.85) -- ++(\tzerodemi + 90:1pt);
    \draw[flow,->]  (\tundemi:0.85)   -- ++(\tundemi - 90:1pt); 
    \node[angle] at (\tm:0.7) {$\theta_{-}$};
    \node[angle,right] at (\tm + 180:0.7) {$\theta_{-}\!+\!\pi$};
    \node[angle] at (\tp:0.7) {$\theta_{+}$};
  \end{scope}
  \draw[line width=2pt,blue!50!red,
   arrows=triangle 90 cap reversed-triangle 90 cap reversed%
   ] (\tzeroM:1.4) arc [radius = 1.4cm,start angle=\tzeroM,end angle=\tunm];
   \draw [decorate, decoration={%
            text along path,
	    text={invariant by both flows},
	    text color=blue!50!red,
	    text align=center}%
	 ]
     (\tunm:1.5) arc [radius = 1.5cm,start angle=\tunm,end angle = \tzeroM];
   \draw (1.1,-0.9) node[rectangle,draw,fill=white,fill opacity=0.8,text opacity=1] {Case (e)};
  
\end{tikzpicture}
}
{
\def\tzerom{50}%
\def\tzeroM{155}%
\def\tunm{30}%
\def\tunM{120}%
\def\tm{0}%
\def\tp{94}%
\def\tzerodemi{40}
\def\tundemi{137.5}
\begin{tikzpicture}[%
  scale=2.4,
  mean/.style={blue!50!red},
  zero/.style={blue!80!black},
  one/.style={red!80!black},
  angle/.style={fill=white,fill opacity = 0.7,text opacity = 1,inner sep=1pt,font={\small}},
  flow/.style={)->,very thick,shorten <= 1pt, shorten >= 1pt},
  alongflow/.style={flow,thin},
  ]
  \draw[very thin] (0,0) circle (1cm);
  \begin{scope}[zero]
    \draw (\tzerom + 180:1.3) -- (\tzerom:1.5);
    \draw (\tzeroM + 180:1.3) -- (\tzeroM:1.5);
    \draw[flow] (\tzeroM:1.15)  arc[radius=1.15cm,start angle=\tzeroM,end angle=\tzerom];
    \draw[flow,->] (-10:1.15)  arc[radius=1.15cm,start angle=-10,end angle=\tzerom ];
    \draw[flow] (\tzeroM :1.15)  arc[radius=1.15cm,start angle=\tzeroM,end angle=190 ];
    \draw[flow,->]  (\tzerodemi:1.15) -- ++(\tzerodemi + 90:1pt);
    \draw[flow,->]  (\tundemi:1.15)   -- ++(\tundemi - 90:1pt); 
    \node[angle,right] at (\tzerom:1.55) {$\theta_{0m}$};
    \node[angle,left] at (\tzeroM:1.55) {$\theta_{0M}$};
  \end{scope}
  \begin{scope}[one]
    \draw (\tunm + 180:1.25) -- (\tunm:1.5);
    \draw (\tunM + 180:1.3) -- (\tunM:1.5);
    \draw[flow,->] (190:1.3)  arc[radius=1.3cm,start angle=190,end angle=\tunM];
    \draw[flow] (\tunm:1.3)  arc[radius=1.3cm,start angle=\tunm,end angle=\tunM];
    \draw[flow] (\tunm:1.3)  arc[radius=1.3cm,start angle=\tunm,end angle= -10];
    \draw[flow,->]  (\tzerodemi:1.3) -- ++(\tzerodemi + 90:1pt);
    \draw[flow,->]  (\tundemi:1.3)   -- ++(\tundemi - 90:1pt); 
    \node[angle,right] at (\tunm:1.55) {$\theta_{1m}$};
    \node[angle,left] at (\tunM:1.55) {$\theta_{1M}$};
  \end{scope}
  \begin{scope}[mean]
    \draw (\tm + 180:1.5) -- (\tm:1.5);
    \draw (\tp + 180:1.2) -- (\tp:1.45);
    \draw[flow] (\tm:0.85)  arc[radius=0.85cm,start angle=\tm ,end angle=\tp];
    \draw[flow] (\tm+180:0.85)  arc[radius=0.85cm,start angle=\tm+180,end angle=\tp];
    \draw[flow,->]  (\tzerodemi:0.85) -- ++(\tzerodemi + 90:1pt);
    \draw[flow,->]  (\tundemi:0.85)   -- ++(\tundemi - 90:1pt); 
    \node[angle] at (\tm:0.7) {$\theta_{-}$};
    \node[angle,right] at (\tm + 180:0.7) {$\theta_{-}\!+\!\pi$};
    \node[angle] at (\tp:0.7) {$\theta_{+}$};
  \end{scope}
  \draw[line width=2pt,blue!50!red,
   arrows=triangle 90 cap reversed-triangle 90 cap reversed%
   ] (\tunM:1.4) arc [radius = 1.4cm,start angle=\tunM,end angle=\tzerom];
   \draw [decorate, decoration={%
            text along path,
	    text={invariant by both flows},
	    text color=blue!50!red,
	    text align=center}%
	 ]
     (\tunM:1.5) arc [radius = 1.5cm,start angle=\tunM,end angle = \tzerom];
   \draw (1.1,-0.9) node[rectangle,draw,fill=white,fill opacity=0.8,text opacity=1] {Case (f)};
\end{tikzpicture}
}

{\small
The outer arrows, in red, represent the flow of $d_1$. 
The middle ones, in blue, represent $d_0$ and the inner ones the averaged flow $d_\lambda$. 
In the two cases, there is a region around $\theta_+$ that is left invariant by both flows.
The regions on each side are unstable and lead back to the invariant region. }

\caption{%
  \label{fig=cases}%
The three flows in cases (e) and (f).} 
\end{figure}

\subsection{Ergodic properties of the angular process}
\label{sec:ergodicAngular}

Since the asymptotic behavior of $R_t=\NRM{X_t}$ depends on the long 
time behavior of the process $(U,I)=(e_\Theta,I)$, let us briefly study its 
ergodicity (recurrent and transient points, number of invariant measures...).

Firstly, remark  that when Assumption~\ref{ass:lambda} is satisfied there 
exists $\varepsilon>0$ such that 
\begin{itemize}
\item the points 
$\BRA{(\theta,i)\,:\, \theta\in (\theta_--\varepsilon,\theta_-+\varepsilon,),
\ i=0,1}$
lead with positive probability to $(\theta_+,j)$ and $(\theta_+-\pi,j)$, $j=0,1$,
\item the points 
$\BRA{(\theta,i)\,:\, \theta\in ( \theta_-+\pi-\varepsilon,\theta_-+\pi+\varepsilon,),
\ i=0,1}$
lead with positive probability to $( \theta_+,j)$ and $(\theta_++\pi,j)$, $j=0,1$.
\end{itemize}
Thus if one of the sets $(\theta_--\varepsilon,\theta_-+\varepsilon)\times\{0,1\}$
or $( \theta_-+\pi-\varepsilon,\theta_-+\pi+\varepsilon)\times\{0,1\}$
  is attained with positive probability starting from $(\theta_+,0)$, then the Markov process $(U_t, I_t)$ on the circle is recurrent. This is the case in the situations (a), (b), (c), (d) described above. In these situations the process  $(U_t, I_t)$ 
 is irreducible and has a unique invariant measure.
 
In the cases (e) and (f), $(U_t, I_t)$ has exactly two distinct recurrent classes and two invariant measures supported by two intervals on the circles corresponding to the invariant interval defined above and its symmetric. Let $\mu_\beta$ and 
$\tilde \mu_\beta$ be these two ergodic invariant measures. 
For any 
initial measure $\mu$ on $\torus\times\BRA{0,1}$, 
  \[
    \frac{1}{t} \int_0^t f(U_s,I_s) ds \xrightarrow[t\to \infty]{\dP_\mu\, a.s.} 
P \int f(u,i) d\mu_\beta(u,i)+(1-P)\int f(u,i) d\tilde\mu_\beta(u,i)
 \]
where $P\in \BRA{0,1}$ is a random variable such that $\dP(P=1)$ is the 
probability that $(U,I)$ reaches the class of $(e_{\theta_+},0)$ when the 
law of $(U_0,I_0)$ is $\mu$. Now by symmetry we have 
\[
\int f(u,i) d\tilde\mu_\beta(u,i)=\int f(-u,i) d\mu_\beta(u,i),
\]
so that, if $f(-u,i)=f(u,i)$, in all the cases, we have  
  \[
    \frac{1}{t} \int_0^t f(U_s,I_s) ds \xrightarrow[t\to \infty]{\dP_\mu\, a.s.} 
    \int f(u,i) d\mu_\beta(u,i).
 \]

Finally notice that  the invariant measures are always absolutely continuous with respect 
to $\lambda_\torus\otimes (\delta_0+\delta_1)$ where $\lambda_\torus$ 
is the Lebesgue measure on $\torus$.

\subsection{Many jumps: blow up}
\label{sec:blowup}

In the sequel, $\mu_\beta$ stands for any invariant measure of $(U,I)$ and we identify $u=e_\theta$ with $\theta$.
As  $\mathcal{A}(\theta,i)=\DP{A_i e_\theta, e_\theta}=\mathcal{A}(\theta+\pi,i)$  we get (see the expression~\eqref{eq:rEnFonctionDeTheta}): 
\[
  \frac{1}{t} \log\PAR{R_t/R_0} \xrightarrow[t\to \infty]{a.s} 
  \int\!\mathcal{A}(\theta,i)\,d \mu_\beta(\theta,i).
\] 
Thus, for any probability measure $\nu$ on $\dR^2\times \BRA{0,1}$ such that 
$\nu(\BRA{0}\times\BRA{0,1})=0$, the 
convergence~\eqref{eq:generic-convergence} in 
Theorem~\ref{thm:mainResult} holds with 
\[
\chi(\beta)=  \int\!\mathcal{A}(\theta,i)\,d \mu_\beta(\theta,i).
\] 
In order to prove that $\chi(\beta)$ is positive when $\beta$ is large we use the 
following lemma, which will be proved in Section \ref{sec:concentration}. 
\begin{lem}
  \label{lem:concentration}
  When $\beta$ is large, the invariant measures are concentrated around 
  the stable points $\theta_+$ and $\tilde\theta_+=\theta_++\pi$ of the averaged 
  dynamical system. More precisely, 
  for any $\epsilon>0$, and any neighborhood $K\subset \torus$ of the 
  set $\{\theta_+, \tilde{\theta}_+\}$, there exists a $\beta(K,\epsilon)$ 
  such that, for any $\beta\geq \beta(K,\epsilon)$,
  \[ \mu_\beta\PAR{ K \times \{0,1\} }  \geq 1 - \epsilon.\]
\end{lem}
Thanks to this result, we can now prove:
\[
  \int \mathcal{A}(\theta,i) d\mu_\beta(\theta,i) > 0
\]
for $\beta$ large enough. For $\theta = \theta_+$ or $\theta = \tilde{\theta}_+$, 
we know that
\[
 \int \mathcal{A}(\theta_+,i) d\mu_\beta(\theta,i) 
 = \DP{ A_\lambda e_{\theta_+}, e_{\theta_+}} = \lambda_+ >0. 
 \]
 Moreover $\mathcal{A}(\cdot,i)$ is continuous for $i = 0,1$. 
 Choose $K$, a neighborhood of $\theta_+, \tilde{\theta}_+$, 
 such that 
 \[ \forall (\theta,i) \in K\times\BRA{0,1}, \quad
 \mathcal{A}(\theta,i) \geq \frac{2\lambda_+}{3}.
 \]
 Thanks to Lemma~\ref{lem:concentration}, 
 for $\beta$ large enough, 
 \[ \mu_\beta( K \times \{0,1\}) \geq 
 1 - \frac{\lambda_+}{6 \NRM{\mathcal{A}}_\infty}.\]
 Therefore:
 \begin{align*}
 \ABS{\int \!\mathcal{A}(\theta,i) d\mu_\beta(\theta,i)  - \lambda_+}
 &\leq \int\!\!\ABS{\mathcal{A}(\theta,i) - \mathcal{A}(\theta_+,i)} 
 \ind_{\theta \in K} d\mu_\beta
      +\int\!\! \ABS{\mathcal{A}(\theta,i) - \mathcal{A}(\theta_+,i)} 
      \ind_{\theta \notin K} d\mu_\beta \\
      &\leq \frac{\lambda_+}{3} +  2 \NRM{\mathcal{A}}_\infty 
      \mu_\beta (\bar{K} \times \{0,1\}) \\
      &\leq \frac{2\lambda_+}{3}. 
 \end{align*}
 This shows that $\chi(\beta) \geq \frac{\lambda_+}{3} > 0$. 
 Hence $R_t$ converges a.s.\  to infinity; this concludes the proof of 
 Theorem~\ref{thm:mainResult}. 

\subsection{The invariant measures concentrate near the attractive points}
\label{sec:concentration}
This section is devoted to the proof of Lemma~\ref{lem:concentration}. 
The idea is that the averaged system gets back quickly to the stable points,
so most of the mass of the invariant measure $\mu_\beta$ should be 
located near these stable points. To quantify this attraction to the stable 
points, we find a Lyapunov function, in the following sense. 
\begin{lem}
  \label{lem:lyapImpliqueConcentration}
  Suppose that there exists a function
$(\theta,i)\mapsto f_\beta(\theta,i)$ that satisfies: 
\begin{align}
f_\beta(\theta,i) &\geq a>0,\nonumber\\ 
L_\beta f_\beta(\theta,i)&\leq -\rho f_\beta(\theta,i)+C\ind_\BRA{\theta \in K},
\label{eq:lyapunov} 
\end{align}
Then $\mu_\beta(K) \geq a\rho/C$. 
\end{lem}
\begin{proof}
 Integrating~\eqref{eq:lyapunov} with respect to the invariant 
measure $\mu_\beta$, we get:
\[ 0 = \int L_\beta f_\beta d\mu_\beta \leq -  \rho \int f_\beta d\mu_\beta 
+ C \mu_\beta(K),\]
which proves the result. 
\end{proof}

The Lyapunov function $f_\beta$ will be constructed  
by the classical ``perturbation'' method (for details see e.g.\ \cite{Kushner}). 
We start from a test function $f$ (depending only on $\theta$) 
adapted to the averaged dynamical system driven by $d_\lambda$, 
and build a perturbation $f_\beta = f - \beta^{-1} g$ of this function such that 
$L_\beta f_\beta \approx L_A f$; this perturbed function will 
satisfy the hypotheses of Lemma~\ref{lem:lyapImpliqueConcentration} 
with appropriate constants. 

Let $K$ be a small neighborhood of the stable points $\theta_+$, 
$\tilde\theta_+$ and $\epsilon>0$. There exists a $2\pi$-periodic 
function $f$ that satisfies the following properties:
\begin{enumerate}
  \item $f$ is $\mathcal{C}^2(\dR)$, 
  \item $f(\theta_- ) = f(\tilde{\theta}_-) = 2$, $f(\theta_+) = f(\tilde{\theta}_+) = 1$, 
  \item $f'(\theta_-) = f'(\theta_+) = f'(\tilde{\theta}_+) = f'(\tilde{\theta}_-) = 0$, 
  \item $f''(\theta_-) = - 1$, $f''(\theta_+) = \epsilon$, 
  \item $f$ is monotonous between its critical points.
\end{enumerate}
Notice that, by design, $f$ decreases along the trajectories of the 
averaged system:
\begin{equation}
  \label{eq:fLyap}\forall \theta \in [0,2\pi], \quad 
  L_Af(\theta)=d_\lambda(\theta)f'(\theta) \leq 0.
\end{equation}
Let us define $g$ on $\mathbb{T}\times\BRA{0,1}$ by 
\[
g(\theta,i) = L_A f(\theta) - L_C f(\theta,i)
\]
where $L_C$ is the continuous part of the $L_\beta$ defined 
in \eqref{eq:def-LC-LJ}. One can notice that, for any $\theta\in \mathbb{T}$, 
$i\mapsto g(\theta,i)$ is the solution of the Poisson equation 
\[
L_J g(\theta,\cdot)=L_Cf(\theta,\cdot)-L_Af(\theta)=-g(\theta,\cdot)
\]
since, for any $\theta\in \mathbb{T}$, 
\[
L_A f(\theta) = \int L_Cf(\theta,i) d(\lambda \delta_1+(1-\lambda)\delta_0)(i).
\]
Finally, define $f_\beta$ on $\mathbb{T}\times \BRA{0,1}$ by 
\[
f_\beta(\theta,i) = f(\theta) - \frac{1}{\beta} g(\theta,i). 
\]
Applying the generator, we get
\begin{align*}
L_\beta f_\beta (\theta,i) 
&= L_C f(\theta,i) - \beta^{-1} L_Cg(\theta,i) + \beta L_Jf(\theta,i) 
- L_J g(\theta,i) \\
&= L_A f(\theta) - \beta^{-1} L_C g(\theta,i). 
\end{align*}
The definition of $g$ ensures that 
\begin{equation}\label{eq:lbrf}
  L_\beta f_\beta(\theta,i)  = L_A f(\theta) + \beta^{-1} Rf(\theta,i)
  \quad\text{where}\quad Rf(\theta,i) = L_CL_C f(\theta,i) -  L_C L_A f(\theta,i), 
\end{equation}
with 
\begin{align*}
  L_CL_Cf(\theta,i) &= d_i(\theta)^2 f''(\theta) + d_i(\theta) d_i'(\theta) f'(\theta)\\
  L_CL_Af(\theta,i) &= d_i(\theta)d_\lambda(\theta) f''(\theta) 
  + d_i(\theta) d_\lambda'(\theta)f'(\theta).
\end{align*}
Thus there exists $\bar R_\epsilon$ such that 
for any $(\theta,i)\in\dR\times\BRA{0,1}$, 
$\ABS{Rf(\theta,i)}\leq \bar R_\epsilon$. In particular, 
if $\beta$ is sufficiently large, one can assume that 
\begin{equation}
  \label{eq:fBetaBornee}
\frac{1}{2}\leq 1 - \epsilon \leq  f_\beta(\theta,i)\leq 3.
\end{equation}
Let us prove \eqref{eq:lyapunov} between two critical points $\theta_-<\theta_+$
splitting the interval $[\theta_-,\theta_+]$ in three regions
\[
[\theta_-,\theta_-+l_-],\quad
[\theta_-+l_-,\theta_+-l_+]\quad\text{and}\quad
[\theta_+-l_+,\theta_+]
\] 
where $l_-$ and $l_+$ depend on $f$, $\varepsilon$ and $K$ (but not on $\beta$).

\paragraph{First region.} 
Since $\theta_-$ is a critical point of $d_\lambda$, one has $L_Af(\theta_-)=0$.
Moreover $f'(\theta_-)$ is equal to 0 since $f$ reaches its minimum at $\theta_-$. 
From \eqref{eq:lbrf}, the expressions of $L_CL_Cf$ and $L_CL_Af$, we 
get that 
\[
L_\beta f_\beta(\theta_-,i)=\beta^{-1}Rf(\theta_-,i)
=\beta^{-1} d_i(\theta_-)^2f''(\theta_-) \leq  - \beta^{-1}c_u
\]
where
\begin{equation}\label{eq:cu}
  c_u = \min \PAR{d_0(\theta_-)^2, d_1(\theta_-)^2 }>0.
\end{equation}
By continuity, we can find $l_->0$ (that does not depend on $\beta$) 
such that $Rf(\theta, i) \leq - c_u/2$ for $\theta\in [\theta_-,\theta_-+l_-]$.  
Remembering \eqref{eq:fLyap}, we obtain:
\begin{align}
  \notag
  L_\beta f_\beta(\theta,i)
  &\leq \beta^{-1} Rf(\theta,i) \\
  \notag
  &\leq - \frac{c_u}{2} \beta^{-1} \\
  \label{eq=lyapTop}
  &\leq -\frac{c_u}{6}\beta^{-1} f_\beta(\theta,i),
\end{align}
where the last line follows from \eqref{eq:fBetaBornee}. 

\paragraph{Second region.}
For $\theta\in[\theta_- + l_-,\theta_+ - l_+]$, $\abs{d_\lambda(\theta)}$ 
and $\abs{f'(\theta)}$ are bounded below, so $L_Af(\theta) \leq -\rho$ 
for some $\rho>0$ that does not depend on $\beta$. Since $Rf$ is bounded, 
\[
L_\beta f_\beta \leq -\frac{\rho}{2}
\]
for $\beta$ large enough. Then \eqref{eq=lyapTop} also holds 
when $\beta$ is large.  

\paragraph{Third region.} 
Since $\theta_+$ is a critical point of $d_\lambda$ and an extremum of 
$f$, $L_Af(\theta_+)=0$ and from \eqref{eq:lbrf}
\[
L_\beta f_\beta(\theta_+,i)=\beta^{-1}Rf(\theta_+,i)
=\beta^{-1} d_i(\theta_+)^2f''(\theta_+) \leq  \beta^{-1}c_d\epsilon
\]
where
\begin{equation}\label{eq:cd}
  c_d = \max \PAR{d_0(\theta_+)^2, d_1(\theta_+)^2}. 
\end{equation}
By continuity, we can find $l_+>0$ such that, 
for any $\theta\in [\theta_+ - l_+, \theta_+]$, 
\[
0\leq Rf(\theta,i) \leq 2 c_d\epsilon 
\quad\text{and}\quad 
1\leq f(\theta,i) \leq 1+\epsilon. 
\]
Notice that $l_+$ does not depend on $\beta$.
Without loss of generality, one can assume that $K$ contains 
$[\theta_+ - l_+, \theta_+]$. We use \eqref{eq:fLyap} once more to get, 
 for $\theta\in [\theta_+-l_+,\theta_+]$,  
\begin{align*}
 L_\beta f_\beta(\theta,i)
 &\leq 2 c_d \epsilon \beta^{-1}\\
 &\leq -\frac{c_u}{6} \beta^{-1} f_\beta(\theta,i) 
 + \frac{c_u}{6}\beta^{-1}f_\beta(\theta,i) 
 + 2 c_d \epsilon \beta^{-1} \\
 &\leq -\frac{c_u}{6} \beta^{-1} f_\beta(\theta,i) 
 + \beta^{-1} \PAR{(1+ \epsilon)\frac{c_u}{6}
 + 2 c_d \epsilon}. 
\end{align*}
\paragraph{Conclusion.}
Gathering the three estimates provides \eqref{eq:lyapunov} with:
\[
  a = \min_{\theta,i} f_\beta(\theta,i), \quad
  \rho = \frac{c_u}{6} \beta^{-1}  \quad\text{and}\quad
  C = \beta^{-1} \PAR{(1+ \epsilon)\frac{c_u}{6} + 2 c_d \epsilon}.
\]
By \eqref{eq:fLyap}, $a\geq 1- \epsilon$ when $\beta$ is large. By Lemma
\ref{lem:lyapImpliqueConcentration}, 
\[
  \mu(K) \geq \frac{(1-\epsilon)\rho}{C} 
  =\frac{1-\epsilon}{1+\epsilon + 12 (c_d/c_u) \epsilon}.
\]
This can be arbitrarily close to $1$ if we choose $\epsilon$ small enough. 

\section{Two explicit examples with a phase transition}
\label{sec:example}

In this section we perform a detail study of Examples~\ref{ex:jordan} 
and~\ref{ex:rotations}. It has been pointed out in Section~\ref{sec:ergodicAngular} that 
the angular processes associated to these two examples are of different 
type. The first one has two recurrent classes whereas the second one is 
ergodic. Nevertheless, we are able to get a perfect picture of the asymptotic 
of $\NRM{X_t}$ as a function of $\beta$ for these two examples. As the 
studies are similar we present precisely the analysis of 
Example~\ref{ex:rotations} and we provide more briefly the key expressions
for Example~\ref{ex:jordan}.

\subsection{Example \ref{ex:rotations}}
Let $a$ and $b$ be two positive real numbers, $\lambda=1/2$ and set 
\[
A_0=\begin{pmatrix}
-1 & ab \\
-a/b & -1
\end{pmatrix}
\quad 
A_1=
\begin{pmatrix}
-1 & -a/b \\
ab & -1
\end{pmatrix}
\]
and
\[
A_{1/2}=\frac{A_1+A_0}{2}=
\begin{pmatrix}
-1 & a(b-1/b)/2 \\
a(b-1/b)/2 & -1
\end{pmatrix}.
\]
The eigenvalues of $A_0$ and $A_1$ are equal to $-1\pm ia$ whereas 
the eigenvalues of $A_{1/2}$ are $-1\pm a(b-1/b)/2$. If $a(b-1/b)>2$, \emph{i.e.} 
$b>1+\sqrt{1+a^2}$, the matrix $A_{1/2}$ admits a positive and a 
negative eigenvalue. The associated eigenvectors are $(1,1)$ and $(1,-1)$.
The generator of the process $(\Theta_t,I_t)$ is given by
\[
L_\beta f(\theta,i)=
d_i(\theta)\partial_\theta f(\theta,i)+\frac{\beta}{2}(f(\theta,1-i)-f(\theta,i)),
\]
where 
\begin{align*}
d_0(\theta)&=-a/b\cos^2(\theta)-ab\sin^2(\theta)<0\\
d_1(\theta)&=ab\cos^2(\theta)+a/b\sin^2(\theta)>0.
\end{align*}

\begin{lem}\label{lem:inv}
 The invariant measure $\mu_\beta$ of the angular process is given by 
 \[
\mu_\beta(d\theta,i)=
\frac{1}{C(\beta)}\frac{1}{\ABS{d_i(\theta)}}e^{\beta v(\theta)}
\ind_{[0,2\pi]}(\theta)\,d\theta, 
\]
where 
\begin{equation}
  \label{eq:def-v} 
v(\theta)=
\begin{cases}
\displaystyle{ \frac{1}{2a}(\arctan (b\tan(\theta))-\arctan (b^{-1}\tan(\theta)))}
&\text{if }\displaystyle{\theta\neq \pm\frac{\pi}{2}}, \\
0 &\text{otherwise.} 
\end{cases}
\end{equation}
and
\[
C(\beta)=\int_0^{2\pi}\!\SBRA{\frac{1}{d_1(\theta)}
-\frac{1}{d_0(\theta)}} e^{\beta v(\theta)}\,d\theta.
\]
\end{lem}
\begin{rem}
Notice that $v$ belongs to $\cC^\infty(\mathbb{T})$ and is $\pi$-periodic. 
Moreover, $v'(\theta)=0$ if and only if $\theta=\pm \pi/4+k\pi$. Finally, 
the function $v$ reaches its maximum at $\pi/4+k\pi$ and its minimum 
at $-\pi/4+k\pi$. 
\end{rem}
\begin{proof}[Proof of Lemma \ref{lem:inv}]
If $\mu_\beta$ is an invariant measure for $(\Theta,I)$, then, for any 
smooth function~$f$ on $\mathbb{T}\times \{0,1\}$, one has  
\[
\int_{\mathbb{T}\times \{0,1\}} L_\beta f(\theta,i)d\mu_{\beta}(\theta,i)=0.
\]
Let us look for an invariant measure $\mu_{\beta}$ on $\mathbb{T}\times \{0,1\}$
that can be written as 
\[
\mu_\beta(d\theta,i)=\rho_0(\theta)\ind_0(i)\,d\theta+\rho_1(\theta)\ind_1(i)\,d\theta,
\]
where $\rho_0$ and $\rho_1$ are two smooth and $2\pi$-periodic functions. 
If $f$ does not depend on the discrete variable $i\in\BRA{0,1}$, \emph{i.e.}
 $f(\theta,i)=f(\theta)$, then 
\[
\int_{\mathbb{T}\times \{0,1\}} L_\beta f(\theta)d\mu_{\beta}(\theta,i)
=\int_{\mathbb{T}}\partial_\theta f(\theta)(d_0\rho_0)(\theta) d\theta
+\int_{\mathbb{T}}\partial_\theta f(\theta)(d_1\rho_1)(\theta) d\theta,
\]
and an integration by parts leads to 
\[
\int_{\mathbb{T}\times \{0,1\}} L_\beta f(\theta)d\mu_{\beta}(\theta,i)=
-\int_{\mathbb{T}}f(\theta)
[d_0\rho_0 + d_1\rho_1]'(\theta) d\theta
\]
This ensures that $d_0\rho_0+d_1\rho_1$ must be constant. Let us assume 
that one can find $\rho_0$ and $\rho_1$ such that 
$d_0\rho_0+d_1\rho_1=0$. Now, if 
$f$ is such that $f(\theta,0)=f(\theta)$ et $f(\theta,1)=0$, we get 
\[
\int_{\mathbb{T}\times \{0,1\}}\! L_\beta f(\theta,i)d\mu_{\beta}(\theta,i)
=\int_{\mathbb{T}}\left[d_0(\theta)\partial_\theta f(\theta)
-\frac{\beta}{2} f(\theta)\right]\rho_0(\theta) d\theta
+\int_{\mathbb{T}}\frac{\beta}{2} f(\theta)\rho_1(\theta) d\theta
\]
and, after an integration by parts, 
\[
\int_{\mathbb{T}\times \{0,1\}}\!L_\beta f(\theta,i)d\mu_{\beta}(\theta,i)
=\int_{\mathbb{T}}f(\theta)\left[-(d_0\rho_0)'(\theta)
+\frac{\beta}{2}(\rho_1(\theta)-\rho_0(\theta))\right] d\theta.
\] 
Let us define $\phi=d_0\rho_0$. Then $\rho_0=\frac{\phi}{d_0}$ and 
$\rho_1=-\frac{\phi}{d_1}$. The function $\phi$ is solution of the following 
ordinary differential equation: 
\begin{equation}\label{eq:edo}
 \phi'=-\frac{\beta}{2}\PAR{\frac{1}{d_1}+\frac{1}{d_0}}\phi.
\end{equation}
This equation admits a solution on $\mathbb{T}$ (\emph{i.e.} $2\pi$-periodic) 
since the integral of $\frac{1}{d_1}+\frac{1}{d_0}$ on $[-\pi,\pi]$ is equal to 0.
In fact this is already true on $[-\pi/2,\pi/2]$. Since $d_0$ and $d_1$ are 
explicit trigonometric functions, one can find an explicit 
expression for $\phi$. Notice that 
\begin{align*}
\SBRA{\arctan \PAR{b^{-1}\tan(\theta)}}'
&=\frac{1}{b} \cdot \frac{1+\tan^2(\theta)}{1+\frac{\tan^2(\theta)}{b^2}}
=\frac{1}{b\cos^2(\theta)+\frac{1}{b}\sin^2(\theta)}=\frac{a}{d_1(\theta)}\\
\SBRA{\arctan (b\tan(\theta))}'&=-\frac{a}{d_0(\theta)}.
\end{align*}
The differential equation \eqref{eq:edo} becomes $\phi'=\beta v'\phi$ where 
$v$ is given by \eqref{eq:def-v} and 
its solutions are given by 
\[
\phi=K\exp(\beta v).
\]
This relation provides the expression of $\rho_0$ and $\rho_1$ up to the 
multiplicative constant $K$. Since we are looking for probability measures, 
$K$ is such that 
\[
K\int_{\mathbb{T}}\PAR{\frac{1}{d_0(\theta)}-\frac{1}{d_1(\theta)}}
\phi(\theta)d\theta=1.
\]
Conversely, it is easy to check that the measure given in Lemma~\ref{lem:inv}
is invariant for $L_\beta$. 
\end{proof}

Let us now consider the function $\chi$ given by 
\[
\chi(\beta)=\int\! \mathcal{A}(\theta,i)\,d\mu_\beta(\theta,i).
\] 

\begin{lem}
 The  function $\beta\mapsto \chi(\beta)$ is a $\cC^1$ and monotonous 
 application on $[0,+\infty)$ such that $\chi'$ has the sign of $b^2-1$ and
 \[
\chi(0)=-1,\quad\lim_{\beta\to\infty}\chi(\beta)=\frac{a(b^2-1)}{2b}-1.
\]
\end{lem}
\begin{proof}
From the definition of $A_i$ and $\mathcal{A}$, we get that, for $i\in\BRA{0,1}$, 
\[
\mathcal{A}(\theta,i)=\DP{A_i e_\theta,e_\theta}
=\frac{a(b^2-1)}{2b}\sin(2\theta)-1.
\]
For sake of simplicity, $\mathcal{A}(\theta)$ stands for 
$\mathcal{A}(\theta,0)=\mathcal{A}(\theta,1)$.
Thus, $\chi(\beta)$ is given by 
\[
\chi(\beta)=\int_0^{2\pi}\!\mathcal{A}(\theta)
\tilde\mu_\beta(d\theta),
\]
where 
\[
\tilde \mu_\beta(d\theta)=
\frac{1}{C(\beta)} \PAR{\frac{1}{d_1(\theta)}-\frac{1}{d_0(\theta)}} 
e^{\beta v(\theta)}\ind_{[0,2\pi]}\,d\theta.
\]
Its derivative is given by 
\begin{align*}
\chi'(\beta)&=
\int_0^{2\pi}\!\mathcal{A}(\theta)v(\theta)\tilde \mu_\beta(d\theta)
-\frac{C'(\beta)}{C(\beta)}\int_0^{2\pi}\!\mathcal{A}(\theta)\tilde \mu_\beta(d\theta)
\\
&=\int_0^{2\pi}\!\mathcal{A}(\theta)v(\theta)\tilde \mu_\beta(d\theta)
-\int_0^{2\pi}\!v(\theta)\tilde \mu_\beta(d\theta)
\int_0^{2\pi}\!\mathcal{A}(\theta)\tilde \mu_\beta(d\theta).
\end{align*}
In other words, one has 
\begin{align*}
\chi'(\beta)&=\mathrm{Cov}_{\tilde\mu_\beta}\PAR{\mathcal{A(\cdot)},v(\cdot)}\\
&=\frac{a(b^2-1)}{2b}\mathrm{Cov}_{\tilde\mu_\beta}\PAR{\sin(2\cdot),v(\cdot)}.
\end{align*}
The mean of $\sin(2\cdot)$ with respect to $\tilde \mu_\beta$ is equal to 
$0$. Besides, $\theta \mapsto v(\theta)\sin(2\theta)$ is nonnegative (and non 
constant) on $\torus$. Thus, $\chi'$ has the sign of $b^2-1$. 

If $\beta=0$, one has
\begin{align*}
\chi(0)
&= \frac{1}{C(0)}\int_0^{2\pi}\! \PAR{\frac{a(b^2-1)}{2b}\sin(2\theta)-1}
                        \PAR{\frac{1}{d_1(\theta)}-\frac{1}{d_0(\theta)}} d\theta \\
&=-\frac{1}{C(0)} 
\int_0^{2\pi}\!\PAR{\frac{1}{d_1(\theta)}-\frac{1}{d_0(\theta)}}d\theta
 = -1 < 0.
\end{align*}
Finally, as $\beta$ goes to $\infty$, the probability measure $\nu_\beta$ 
converges to a probability measure concentrated on the points 
$\{\pi/4,5\pi/4,\}$ where $v$ reaches its maximum. We get 
\[
\lim_{\beta\rightarrow +\infty}\chi(\beta)=\frac{a(b^2-1)}{2b}-1.
\]
This concludes the proof.
\end{proof}

\begin{cor}
 If $b>1+\sqrt{1+a^2}$, then there exists $\beta_c\in(0,+\infty)$ such that 
$\chi$ is negative on $(0,\beta_c)$ and positive on $(\beta_c,+\infty)$. 
\end{cor}

\subsection{Example \ref{ex:jordan}}

Let us define $A_0$ and $A_1$ by 
\[
A_0=\PAR{
\begin{array}{cc}
-1 & 2b \\
0  & -1 
\end{array}
}
\quad\text{and}\quad
A_1=\PAR{
\begin{array}{cc}
-1 & 0 \\
2b  & -1 
\end{array}
}
\]
with $b>0$. Then $A_0$ and $A_1$ are two Jordan matrices and 
the eigenvalues of $A_{1/2}$ are given by $-1\pm b$. In this case, 
\[
d_0(\theta)=-2b\sin^2(\theta)\leq 0
\quad\text{and}\quad
d_1(\theta)=2b\cos^2(\theta)\geq 0,
\]
and $(\Theta,I)$ has two recurrent classes 
\[
C_1=\BRA{(\theta,i)\,:\, \theta\in(0,\pi/2),\ i=0,1} 
\quad\text{and}\quad
C_2\BRA{(\theta,i)\,:\, \theta\in(\pi,3\pi/2),\ i=0,1}.
\]
It can be shown, following the lines of the previous section that the 
ergodic invariant measure $\mu_\beta$ of the angular process on 
$C_1$ is given by 
 \[
\mu_\beta(d\theta,i)=
\frac{1}{C(\beta)}\cdot \frac{1}{\ABS{d_i(\theta)}}e^{\beta v(\theta)}
\ind_{(0,\pi/2)}(\theta)\,d\theta, 
\]
where 
\[
v(\theta)=-\frac{1}{2b\sin(2\theta)}
\quad\text{and}\quad
C(\beta)=\frac{2}{b}\int_0^{\pi/2}\!\frac{1}{\sin^2(2\theta)}
e^{\beta v(\theta)}\,d\theta.
\]
Moreover, for any $\beta>0$, 
\[
\chi(\beta)=-1+\frac{1}{C(\beta)}\int_0^{\pi/2}\!\frac{2}{\sin(2\theta)} e^{\beta v(\theta)}
\,d\theta. 
\]
In particular, the  function $\beta\mapsto \chi(\beta)$ is a $\cC^1$ increasing 
application on $[0,+\infty)$ such that 
 \[
\chi(0)=-1,\quad\lim_{\beta\to\infty}\chi(\beta)=-1+b.
\]

\begin{cor}
  If $b>1$, then there exists $\beta_c\in(0,+\infty)$ such that 
$\chi$ is negative on $(0,\beta_c)$ and positive on $(\beta_c,+\infty)$. 
\end{cor}

\section{Application to matrix products}
\label{sec:matrices}
The process studied in the preceding sections is linked to some products 
of random matrices. Let us consider the embedded chain of our process 
defined by the sequence of the positions of the process $X$ at the times 
when the second coodinate $I$ changes, that is the positions at the times 
when one changes the flow. The jump times are given by sums of 
independent random variables with exponential law of parameters 
$\lambda_0\beta$ and $\lambda_1\beta$. To study this embedded chain is 
to study the linear images of vectors by products of independent random 
matrices which distributions are the image laws of  exponential law of 
parameter~1 by the two mappings  
\[
s\mapsto \exp((s/\beta\lambda_0) A_0)
\quad \text{and}\quad 
s\mapsto \exp((s/\beta\lambda_1) A_1).
\]
Let us denote ${(T_k)}_{k\geq 0}$ the sequence of the jump times of 
the second coordinate (with the convention $T_0=0$) and ${(Z_k)}_{k\geq 0}$ 
the sequence of the positions of $X$ at these times:
\[
Z_k=X_{T_k}. 
\]
The embedded chain and the process ${(X_t)}{t\geq 0}$ are linked as follows. 
For $t\in]T_k,T_{k+1}]$ one has :
\[
X_t=\exp\PAR{\frac{t-T_k}{\beta\lambda_{i_k}}A_{i_k}}Z_k,
\]
where $i_k$ is 0 or 1 depending on the evenness of $k$.
Thus, 
\[
Z_k=U_kU_{k-1}\ldots U_1X_0
\quad\text{where}\quad
U_l=\exp\PAR{\frac{T_l-T_{l-1}}{\beta\lambda_{i_{l-1}}}A_{i_{l-1}}}.
\]
For example we can fix that $i_0=0$, which means that at time 0, $X$ is driven 
by the vector field $x\mapsto A_0 x$.

Let $e^{(1)}$ and $e^{(2)}$ be the element of the canonical basis of 
$\mathbb{R}^2$, $X_t^{(1)}$ and $X_t^{(2)}$ the processes starting 
from $e^{(1)}$ and $e^{(2)}$ respectively. From the equality
\[
X_t^{(1)}=\exp\PAR{\frac{t-T_k}{\beta\lambda_{i_k}}A_{i_k}}
U_kU_{k-1}\ldots U_1e^{(1)},
\]
we get
\begin{align*}
\|U_kU_{k-1}\ldots U_1\|&\geq \|U_kU_{k-1}\ldots U_1e^{(1)}\|  \\
&\geq \|\exp(-((t-T_k)/\beta\lambda_{i_k})A_{i_k})X_t^{(1)}\|  \\
&\geq \|\exp(((t-T_k)/\beta\lambda_{i_k})A_{i_k})\|^{-1}\|X_t^{(1)}\|.
\end{align*}
On the other hand, for $t\in]T_k,T_{k+1}]$, we have
\begin{align*}
\|U_kU_{k-1}\ldots U_1\|
&\leq \|U_kU_{k-1}\ldots U_1e^{(1)}\|+\|U_kU_{k-1}\ldots U_1e^{(2)}\|\\
&=\sum_{j=1}^2\|\exp(-((t-T_k)/\beta\lambda_{i_k})A_{i_k})\|\|X_t^{(j)}\|\\
&\leq 2\|\exp(-((t-T_k)/\beta\lambda_{i_k})A_{i_k})\|
\max(\|X_t^{(1)}\|,\|X_t^{(2)}\|).
\end{align*}
According to Theorem~\ref{thm:mainResult} almost surely both limits
\[
\lim_{t\rightarrow\infty}\frac{1}{t}\log \|X_t^{(1)}\|\
\quad\text{and} \quad 
\lim_{t\rightarrow\infty}\frac{1}{t}\log \|X_t^{(2)}\|
\]
exist and are equal to $\chi(\beta)$. Moreover, almost surely, the 
ratio $(t-T_k)/ t$ tends to 0 and, as $T_k$ is the sum of independent 
random variables of parameter $\lambda\beta$ and $(1-\lambda)\beta$, 
the strong law of large numbers gives 
\[
\frac{T_{2k}}{2k}\xrightarrow[k\rightarrow \infty]{}\frac{1}{2\lambda\beta}
+\frac{1}{2(1-\lambda)\beta}
=\frac{1}{2\lambda(1-\lambda)\beta},
\]
so that $T_k/k$ almost surely tends toward 
$\PAR{2\lambda(1-\lambda)\beta}^{-1}$. Putting things together we get 
that, almost surely,
\[
\lim_{k\rightarrow\infty}\frac{1}{k}\log \|U_kU_{k-1}\ldots U_1\|
=\frac{\chi(\beta)}{2\lambda(1-\lambda)\beta}.
\]
In particular this limit has the same sign as $\chi(\beta)$, it is negative for 
small $\beta$ and positive for large $\beta$.

This does give an example of a product of contracting independent matrices 
with a positive Lyapunov exponent but in this case the matrices $({U_k})_k$ 
do not have the same distribution : it depends on the evenness of $k$. If we 
group the $U_k$ by 2 we get a product of independent matrices with the same 
distribution but they are not always contracting: some matrices in the image of
\[
(s,t)\mapsto \exp\PAR{\frac{t}{\beta\lambda_1}A_1}
\exp\PAR{\frac{s}{\beta\lambda_0}A_0}
\]
are hyperbolic. 

So let us slightly modifiy the process we began with. When the second 
coordinate is $i\in\{0,1\}$, at each date given by the sum of independent 
random variables with exponential law of parameter $\lambda_i\beta$ 
one chooses independently with probability 1/2 to keep the flow $i$ or with
probability 1/2 to flip to the flow $1-i$. As an independent geometric random 
sum of exponential independent random variables is still an exponential 
random variable, in continuous time, this modification is simply a change 
of parameter $\beta$ (replaced par $\beta/2$). The embedded chain defined 
by the position at times given by (not the changes of flow but) the sums 
of exponential random variables, also corresponds to a products of 
independent random matrices, and this time, all matrices considered 
are contracting.
 
Let $(D_k)$ denotes the sequence of dates considered in this case. 
It is a sum of $k$ independent exponential variables of 
parameters~$\beta\lambda_0$ and $\beta\lambda_1$ and, almost 
surely, asymptotically, half of them are of parameter $\beta\lambda_0$, 
half of them of parameter $\beta\lambda_1$. So that, as before, $D_k/k$ 
almost surely tends to $\PAR{2\lambda(1-\lambda)\beta}^{-1}$. These 
remarks and the preceding computation give the following proposition.
\begin{prop}
Let $A_0$ and  $A_1$ two matrices such that Assumption~\ref{ass:lambda} 
is satisfied. Let ${(V_k)}_{k\geq 1}$ be a sequence of independent matrices with 
distribution given by the half sum of the image measures of the exponential 
law of parameter 1 by the two mappings
\[
s\mapsto \exp\PAR{\frac{s}{\beta\lambda_0}A_0}
\quad\text{and}\quad
t\mapsto \exp\PAR{\frac{t}{\beta\lambda_1}A_1}.
\]
Then almost surely, one has
\[
\lim_{k\rightarrow\infty}\frac{1}{k}\log \|V_kV_{k-1}\ldots V_1\|
=\frac{\chi(\beta/2)}{2\lambda(1-\lambda)\beta},
\]
and if $\beta$ is sufficiently large this limit is positive.
\end{prop}
Thus we have obtained examples of product of random independent 
identically distributed matrices, all contracting, with a positive 
Lyapounov exponent.

\bigskip
\noindent
\textbf{Acknowledgements.} FM and PAZ thank MB for his kind hospitality 
and his coffee breaks. We acknowledge financial support from the Swiss National Foundation Grant  FN 200021-138242/1
and the French ANR projects EVOL and ProbaGeo.

\addcontentsline{toc}{section}{\refname}%
{ \footnotesize
\bibliography{exemple}
\bibliographystyle{amsplain}
}

\bigskip

{\footnotesize
\noindent Michel~\textsc{Benaïm}, 
 e-mail: \texttt{michel.benaim(AT)unine.ch}
 \medskip

\noindent\textsc{Institut de Math\'ematiques, Universit\'e de Neuch\^atel, 
11 rue \'Emile Argand, 2000 Neuch\^atel, Suisse.}

\bigskip

\noindent Stéphane~\textsc{Le Borgne}, 
 e-mail: \texttt{stephane.leborgne(AT)univ-rennes1.fr}

 \medskip

  \noindent\textsc{UMR 6625 CNRS Institut de Recherche Math\'ematique de
    Rennes (IRMAR) \\ Universit\'e de Rennes 1, Campus de Beaulieu, F-35042
    Rennes \textsc{Cedex}, France.}

\bigskip

 \noindent Florent~\textsc{Malrieu},
 e-mail: \texttt{florent.malrieu(AT)univ-rennes1.fr}

 \medskip

  \noindent\textsc{UMR 6625 CNRS Institut de Recherche Math\'ematique de
    Rennes (IRMAR) \\ Universit\'e de Rennes 1, Campus de Beaulieu, F-35042
    Rennes \textsc{Cedex}, France.}

\bigskip

\noindent Pierre-Andr\'e~\textsc{Zitt}, 
e-mail: \texttt{Pierre-Andre.Zitt(AT)u-bourgogne.fr}

\medskip 

\noindent\textsc{UMR 5584 CNRS Institut de Math\'ematiques de Bourgogne,\\
Universit\'e de Bourgogne, UFR Sciences et Techniques,\\
9 avenue Alain Savary -- BP 47870, 
21078 Dijon Cedex, France
}
}
\end{document}